\documentclass[a4paper,11pt,twoside]{article}

\usepackage{geometry}
\usepackage[utf8]{inputenc}
\usepackage{lmodern}
\usepackage[T1]{fontenc}
\usepackage{dsfont}
\usepackage{tikz}
\usepackage{amsmath,amssymb,amsthm,empheq,cases}
\usepackage{graphicx}
\usepackage{url}
\usepackage{hyperref}
\hypersetup{colorlinks,
            citecolor=red, 
            filecolor=black,
            linkcolor=blue,
            urlcolor=black}

\def \dis {\displaystyle}

\def \NN {\mathbb N}

\def \RR {\mathbb R}
\def \CC {\mathbb C}

\def \E {\mathcal{E}}

\def \L {\mathcal{L}}

\def \R {\mathcal{R}}

\def \ecart {\noalign{\medskip}}

\theoremstyle{definition}
\newtheorem{Th}{Theorem}[section]
\newtheorem{Prop}[Th]{Proposition}
\newtheorem{Lem}[Th]{Lemma}
\newtheorem{Cor}[Th]{Corollary}
\newtheorem{Def}[Th]{Definition}
\newtheorem{Rem}[Th]{Remark}

\def \refT #1{Theorem~\ref{#1}}
\def \refL #1{Lemma~\ref{#1}}
\def \refC #1{Corollary~\ref{#1}}
\def \refP #1{Proposition~\ref{#1}}
\def \refR #1{Remark~\ref{#1}}

\title{A biharmonic transmission problem in $L^p$-spaces}

\author{Alexandre Thorel \\ \ecart
{\footnotesize Normandie Univ, UNIHAVRE, LMAH, FR-CNRS-3335, ISCN, 76600 Le Havre, France.} \\ \ecart alexandre.thorel@univ-lehavre.fr}

\date{}

\begin{document}
\maketitle

\begin{abstract}
In this work we study, by a semigroup approach, a transmission problem based on biharmonic equations with boundary and transmission conditions, in two juxtaposed habitats. We give a result of existence and uniqueness of the classical solution in $L^p$-spaces, for $p \in (1,+\infty)$, using analytic semigroups and operators sum theory in Banach spaces. To this end, we invert explicitly the determinant operator of the transmission system in $L^p$-spaces using the $\E_{\infty}$-calculus and the Dore-Venni sums theory. 


\end{abstract}

\section{Introduction}

In this work, we consider a system of linear biharmonic equations posed on two juxtaposed domains and coupled through transmission conditions at the interface. Throughout the paper we shall impose the continuity of the flux, of the dispersal and of its flux across the interface. Using an operator approach, we investigate the existence, uniqueness as well as maximal $L^p$-regularity for such a problem.

Transmission problems arise in various applicative fields including engineering, physics and biology. Here, we refer the reader for instance to \cite{JMAA 1}, \cite{hassan} or \cite{JMAA 2} for applications in plate theory, to \cite{JEE 2}, \cite{perfekt} or \cite{JDE} for applications in electromagnetism and to \cite{menad}, \cite{maelis} or \cite{DCDS-B} for other applications in population dynamics. Let us also mention that mathematical models involving biharmonic operators also arise in various fields such as elasticity for instance see \cite{elasticity 1}, \cite{elasticity 2} or \cite{IJPAM}, electrostatic see \cite{CV 1}, \cite{electrostatic 1} or \cite{electrostatic 2}, plate theory \cite{JMAA 1}, \cite{hassan} or \cite{JMAA 2} or population dynamics \cite{cohen-murray}, \cite{LMMT}, \cite{LLMT} or \cite{ochoa}.

In this work, we consider an $n$-dimensional (with $n\geqslant 2$) straight cylinder of the form 
$$\Omega=(a,b)\times\omega,$$ 
where $a<b$ are two given real numbers while the section $\omega\subset \RR^{n-1}$ denotes a smooth bounded domain. This cylinder $\Omega$ is split into two (open) sub-cylinders $\Omega_\pm$ and an interface $\Gamma$ given for some $\gamma \in (a,b)$ by
$$\Omega_-=(a,\gamma)\times \omega,\quad \Omega_+=(\gamma,b)\times\omega\quad \text{and} \quad \Gamma=\{\gamma\}\times \omega,$$
so that $\Omega=\Omega _{-}\cup \Gamma \cup\Omega _{+}$.

We consider the following biharmonic equations,
\begin{equation*}
(EQ_{pde})\left\{ 
\begin{array}{c}
k_{-}\Delta ^{2}u_{-} = g_{-}\text{ \ \ in \ \ }\Omega _{-}
\\ \ecart
\ k_{+}\Delta ^{2}u_{+} = g_{+}\text{ \ \ in \ \ }\Omega _{+},%
\end{array}%
\right.
\end{equation*}
where $g_- \in L^p(\Omega_-)$, $g_+ \in L^p(\Omega_+)$ are given and $k_+,k_- >0$. 

We denote by $(x,y)$ the spatial variables with $x\in (a,b)$ and $y\in \omega $. Then, we consider the following conditions on $\partial \Omega \setminus \Gamma$, the lateral boundary of $\Omega$, 
\begin{equation*}
(BC_{pde})\left\{\hspace{-0.2cm} \begin{array}{l}
(1)\left\{\begin{array}{rcllrcll}
u_{-}(x,\zeta)&\hspace{-0.2cm}=&\hspace{-0.2cm}0, &x\in (a,\gamma),~ \zeta \in \partial\omega, & u_{+}(x,\zeta)&\hspace{-0.2cm}=&\hspace{-0.2cm}0, & x\in (\gamma ,b),~ \zeta \in \partial\omega \\ 
\dis \Delta u_- (x,\zeta)&\hspace{-0.2cm}=&\hspace{-0.2cm}0, & x\in (a,\gamma),~ \zeta \in \partial\omega, & \dis \Delta u_+ (x,\zeta)&\hspace{-0.2cm}=&\hspace{-0.2cm}0, & x\in (\gamma ,b),~ \zeta \in \partial\omega
\end{array}\right. \\ \ecart
(2) \left\{\begin{array}{rclrrclr}
u_{-}(a,y)&=&\varphi _{1}^{-}(y), & u_{+}(b,y)&=&\varphi _{1}^{+}(y), &y \in \omega \\ 
\dis\frac{\partial u_{-}}{\partial x}(a,y) &=& \varphi_{2}^{-}(y), & \dis\frac{\partial u_{+}}{\partial x}(b,y) &=& \varphi _{2}^{+}(y), & y\in \omega,
\end{array}\right.\end{array}\right.
\end{equation*}
where $\varphi_1^\pm$ and $\varphi_2^\pm$ will be given in appropriated spaces. The system is coupled on the interface $\Gamma$ where we impose the following continuity conditions,
\begin{equation*}
(TC_{pde})\left\{ 
\begin{array}{rcll}
u_{-}&=&u_{+}&\text{on \ }\Gamma \\ 
\dis\frac{\partial u_{-}}{\partial x}&=&\dis\frac{\partial u_{+}}{\partial x} &\text{on \ }\Gamma \medskip \\ 
k_{-}\Delta u_{-}&=&k_{+}\Delta u_{+} & \text{on \ }\Gamma \\ 
\dis k_{-} \frac{\partial \Delta u_{-}}{\partial x} & = & \dis k_{+} \frac{\partial \Delta u_{+}}{\partial x} & \text{on \ }\Gamma.
\end{array}%
\right.
\end{equation*}

Let us now explain, for instance, in population dynamics framework, the boundary and transmission conditions.

The first line of $(BC_{pde})-(1)$, means that the individuals could not lie on the boundaries $(a,b) \times \partial \omega $, because, for instance, they die or the edge is impassable. The second line means that there is no dispersal in the normal direction. We deduce that the dispersal vanishes on $(a,b) \times \partial \omega$.

In $(BC_{pde}) - (2)$, the population density and the flux are given on $\{a,b\} \times \omega$. This means that the habitats are not isolated.

In $(TC_{pde})$, the two first transmission conditions mean the continuity of the density and its flux at the interface, while the two second express, in some sense, the continuity of the dispersal and its flux at $\Gamma$.

This work is a natural continuation of that done in \cite{thorel}. Moreover, it completes the study realized in \cite{LLMT} where the authors have considered equations
\begin{equation*}
k_\pm \Delta^2 u_\pm - l_\pm \Delta u_\pm = f_\pm \text{ \ \ in \ \ }\Omega_\pm
\end{equation*}
and on the interface $\Gamma$, the last condition of $(TC_{pde})$ is replaced by
$$\frac{\partial }{\partial x}\left( k_+ \Delta u_+ - l_+ u_+\right) 
\text{\ \ on \ \ }\Gamma.$$
In the present work, $l_\pm = 0$, which must be treated differently, since the proof in \cite{LLMT} uses the representation formula obained in \cite{LMMT} which is not defined when $l_\pm = 0$. Thus, we use the representation formula obtained in \cite{thorel} which is different and does not corresponds to the representation formula in \cite{LMMT} when the limit of $l_\pm$ tends to $0$. Therefore, even if the proof follows the same steps than the one in \cite{LLMT}, the calculus are different and cannot be deduce from those in \cite{LLMT}, since they use a different representation formula.

Note that, due to the transmission conditions, we can not obtain a solution $u \in W^{4,p}(\Omega)$ in all $\Omega$. We only can obtain a solution $u$ such that $u_{|_{\Omega_-}} \in W^{4,p}(\Omega_-)$ and $u_{|_{\Omega_+}} \in W^{4,p}(\Omega_+)$.

The paper is organized as follows. 

First, in section~\ref{Section Operational formulation}, we recall the PDE transmission problem $(P_{pde})$ and we rewrite it under operational form. Then, in section~\ref{Section Statement}, we recall some definitions about BIP operator and interpolation spaces. We give our hypotheses and their consequences. We present our main result in \refT{Th principal} and \refC{CorAppli} which states existence and uniqueness of the solution of problem $(P_{pde})$ that is $(EQ_{pde})-(BC_{pde})-(TC_{pde})$ quoted above. In section~\ref{Section Proof prel res}, we state technical results which allow us to prove our main result. Section~\ref{Section proof main res}, is devoted to the proof of \refT{Th principal}.

\section{Operational formulation}\label{Section Operational formulation}

In this section, we first recall the PDE problem $(P_{pde})$ composed by $(EQ_{pde})-(BC_{pde})-(TC_{pde})$. Then, we define the Laplace operator $A_0$ and we use it to rewrite $(P_{pde})$. Note that this operational form is a vector values problem. Finally, we generalize this problem replacing $A_0$ by a more general operator $A$. We consider the following problem
\begin{equation*}
(P_{pde})\left\{ \begin{array}{lcllcll}
k_- \Delta^2 u_- & = & g_- &&&& \text{in } \Omega_- \\ 
k_+ \Delta^2 u_+ & = & g_+ &&&& \text{in } \Omega_+ \\ \ecart

u_{-}(x,\zeta) & = & 0, & \Delta u_- (x,\zeta) & = & 0, & x\in~(a, \gamma),~\zeta \in \partial\omega \\ 
u_{+}(x,\zeta) & = & 0, & \Delta u_+ (x,\zeta) & = & 0, & x\in~(\gamma ,b),~\zeta \in \partial\omega  \vspace{0.1cm}\\ 

u_-(a,y) & = &\varphi_1^-(y), &\dis u_+(b,y) & = & \varphi_1^+(y), & y\in \omega \vspace{0.1cm}\\ 
\dis \frac{\partial u_-}{\partial x}(a,y) & = & \varphi_2^-(y), &\dis \frac{\partial u_+}{\partial x}(b,y) & = & \varphi_2^+(y), & y \in \omega \\ \ecart

u_- &=& u_+ &&&& \text{on } \Gamma \vspace{0.1cm}\\ 
\dis\frac{\partial u_-}{\partial x} &=& \dfrac{\partial u_+}{\partial x} &&&& \text{on } \Gamma \vspace{0.15cm}\\ 
k_- \Delta u_- & = & k_+ \Delta u_+ &&&& \text{on } \Gamma \vspace{0.1cm} \\
k_- \dfrac{\partial \Delta u_-}{\partial x} & = & k_+ \dfrac{\partial \Delta u_+}{\partial x} &&&& \text{on }\Gamma.
\end{array}\right.
\end{equation*}
Let us define $A_0$, the Dirichlet Laplace operator in $\RR^{n-1}$, $n \in \NN\setminus\{0,1\}$, as follows
\begin{equation}\label{def A0}
\left\{\begin{array}{l}
D(A_0) := \{\psi \in W^{2,p}(\omega) : \psi = 0 ~\text{on} ~\partial \omega\} \\ \ecart
\forall \psi \in D(A_0), \quad A_0 \psi = \Delta_y \psi.
\end{array}\right.
\end{equation}
Thus, using operator $A_0$, problem $(P_{pde})$ becomes the following vector valued problem
\begin{equation*}
\left\{\begin{array}{l}
\begin{array}{lll}
\dis u_-^{(4)}(x) + 2A_0 u''_- (x) + A^2_0 u_-(x) & = & f_-(x), \quad \text{for a.e. } x \in (a,\gamma) \\ \ecart
\dis u_+^{(4)}(x) + 2A_0 u''_+ (x) + A^2_0 u_+(x) & = & f_+(x), \quad \text{for a.e. } x \in (\gamma, b)
\end{array} \\ \ecart
\begin{array}{llllll}
u_- (a) & = & \varphi_1^-, \quad u_+ (b) & = & \varphi_1^+ \\ \ecart
u'_-(a) & = & \varphi_2^-, \quad u'_+(b) & = & \varphi_2^+ \\ \ecart
u_-(\gamma) & = & u_+(\gamma) \\ \ecart
u'_-(\gamma) & = & u'_+(\gamma) 
\end{array} \\ \ecart
\begin{array}{lll} 
k_- u''_-(\gamma)  ~+  k_- A_0 u_-(\gamma) & = &  k_+ u''_+(\gamma) ~ +  k_+ A_0 u_+(\gamma)\\ \ecart
k_- u_-^{(3)}(\gamma)  +  k_- A_0 u'_-(\gamma) & = & k_+ u_+^{(3)}(\gamma)  + k_+ A_0 u'_+(\gamma),
\end{array}\end{array}\right.
\end{equation*}
where 
$$f_- \in L^p(a,\gamma;L^p(\omega)), \quad f_+ \in L^p(\gamma,b;L^p(\omega)) \quad \text{and}\quad p \in (1, +\infty),$$
with
$$u_\pm(x):=u(x,\cdot) \quad\text{and} \quad f_\pm(x):=g_\pm(x,\cdot)/k_\pm.$$ 

Note that the boundary conditions on $\partial\omega$ in $(P_{pde})$ do not appear in the previous system since they are already included in the domain of $A_0$. Thus, a classical solution of this problem satisfies the boundary conditions on $\partial\omega$ in $(P_{pde})$.

Then, using operator $(A,D(A))$ instead of $(A_0,D(A_0))$ and $X$ instead of $L^p(\omega)$, we can write that $f_- \in L^p(a,\gamma;X)$ and $f_+ \in L^p(\gamma,b;X)$. 

Moreover, in all the sequel, we will study the following more general transmission problem :
\begin{equation*}
\text{(P)} \left\{\begin{array}{l}
(EQ)\left\{\begin{array}{lll}
\dis u_-^{(4)}(x) + 2A u''_- (x) + A^2 u_-(x) & = & f_-(x), \quad \text{for a.e. } x \in (a,\gamma) \\ \ecart
\dis u_+^{(4)}(x) + 2A u''_+ (x) + A^2 u_+(x) & = & f_+(x), \quad \text{for a.e. } x \in (\gamma, b)
\end{array}\right.\\ \ecart
(BC) \left\{\begin{array}{llllll}
u_- (a) & = & \varphi_1^-, \quad u_+ (b) & = & \varphi_1^+ \\ \ecart
u'_-(a) & = & \varphi_2^-, \quad u'_+(b) & = & \varphi_2^+
\end{array}\right.\\ \ecart
(TC) \left\{\begin{array}{lll}
\hspace*{-0.15cm}\begin{array}{lll}
u_-(\gamma)  & = & u_+(\gamma) \\ \ecart
u'_-(\gamma) & = & u'_+(\gamma) 
\end{array}\\ \ecart 
k_- u''_-(\gamma) ~ +  k_- Au_-(\gamma)& = & k_+ u''_+(\gamma) ~ +  k_+ Au_+(\gamma) \\ \ecart
k_- u_-^{(3)}(\gamma)  +  k_- Au'_-(\gamma) & = & k_+ u_+^{(3)}(\gamma)  + k_+ Au'_+(\gamma).
\end{array}\right.\end{array}\right.
\end{equation*}
The transmission conditions $(TC)$ will be split into 
$$(TC1)\left\{\begin{array}{lll}
u_-(\gamma) & = & u_+(\gamma) \\ \ecart
u'_-(\gamma) & = & u'_+(\gamma),
\end{array}\right.$$
and
$$(TC2)\left\{\begin{array}{lll}
k_- u''_-(\gamma) ~ +  k_- Au_-(\gamma)& = & k_+ u''_+(\gamma) ~ +  k_+ Au_+(\gamma) \\ \ecart
k_- u_-^{(3)}(\gamma)  +  k_- Au'_-(\gamma) & = & k_+ u_+^{(3)}(\gamma)  + k_+ Au'_+(\gamma).
\end{array}\right.$$
Note that $(TC2)$ is well defined from \refL{remtrace}, see Section \ref{sect interp} below. 

We will search a classical solution of problem (P), that is a solution $u$ such that 
\begin{equation}\label{u-=u(a,gamma),u+=u(gamma,a)}
\left\{\begin{array}{ll}
u_-:= u_{|(a,\gamma)} \in W^{4,p}(a,\gamma;X)\cap L^p(a,\gamma;D(A^2)), & u''_- \in L^p(a,\gamma;D(A)), \\ \ecart

u_+:= u_{|(\gamma,b)} \in W^{4,p}(\gamma,b;X)\,\cap L^p(\gamma,b;D(A^2)), & u''_+ \in L^p(\gamma,b;D(A)),
\end{array}\right.
\end{equation}
and which satisfies $(EQ)-(BC)-(TC)$. 

\section{Assumptions, consequences and statement of results}\label{Section Statement}

\subsection{The class BIP$\mathbf{(X,\theta)}$}

Throughout the article, $(X,\|\cdot\|)$ is a complex Banach space. We give some definitions about UMD spaces, sectorial operators and BIP operators.

\begin{Def}
A Banach space $X$ is a UMD space if and only if for all $p\in(1,+\infty)$, the Hilbert transform is bounded from $L^p(\RR,X)$ into itself (see \cite{bourgain} and \cite{burkholder}).
\end{Def}
\begin{Def}
A closed linear operator $T_1$ is called sectorial of angle $\theta \in [0,\pi)$ if
$$
\begin{array}{l}
i)\quad \sigma(T_1)\subset \overline{S_{\theta}},\\ \ecart
ii)\quad\forall~\theta'\in (\theta,\pi),\quad \sup\left\{\|\lambda(\lambda\,I-T_1)^{-1}\|_{\L(X)} : ~\lambda\in\CC\setminus\overline{S_{\omega'}}\right\}<+\infty,
\end{array}
$$
 where
\begin{equation}\label{defsector}
S_\omega\;:=\;\left\{\begin{array}{lll}
\left\{ z \in \CC : z \neq 0 ~~\text{and}~~ |\arg(z)| < \omega \right\} & \text{if} & \omega \in (0, \pi), \\ \ecart
\,(0,+\infty) & \text{if} & \omega = 0,
\end{array}\right.
\end{equation} 
see \cite{haase}, p. 19, such an operator is noted $A \in \text{Sect}\,(\omega)$.
\end{Def}
\begin{Rem}
From \cite{komatsu}, p. 342, we know that any injective sectorial operator~$T_1$ admits imaginary powers $T_1^{is}$, $s\in\RR$, but, in general, $T_1^{is}$ is not bounded.
\end{Rem} 
\begin{Def}
Let $\theta \in[0, \pi)$. We denote by BIP$(X,\theta)$, the class of sectorial injective operators $T_1$ such that
\begin{itemize}
\item[] $i) \quad ~~\overline{D(T_1)} = \overline{R(T_1)} = X,$

\item[] $ii) \quad ~\forall~ s \in \RR, \quad T_1^{is} \in \L(X),$

\item[] $iii) \quad \exists~ C \geq 1 ,~ \forall~ s \in \RR, \quad ||T_1^{is}||_{\L(X)} \leq C e^{|s|\theta}$,
\end{itemize}
see~\cite{pruss-sohr}, p. 430.
\end{Def}
\begin{Rem}
From \cite{haase}, proof of Proposition 3.2.1, c), p. 71, $\overline{D(T_1) \cap R(T_1)} = X$.
\end{Rem} 

\subsection{Interpolation spaces}\label{sect interp}

Here we recall a definition given, for instance, in \cite{da prato-grisvard}, \cite{grisvard}, \cite{lions-peetre} or in \cite{triebel} and a result from \cite{grisvard} concerning real interpolation spaces. 
\begin{Def}
Let $T_2 : D(T_2) \subset X \longrightarrow X$ be a linear operator such that
\begin{equation}\label{def T2}
(0,+\infty) \subset \rho(T_2)\quad \text{and}\quad \exists~C>0:\forall~t>0, \quad \|t(T_2-tI)^{{\mbox{-}}1}\|_{\L(X)} \leqslant C.
\end{equation}
Let $k \in \NN \setminus \{0\}$, $\theta \in (0,1)$ and $q \in [1,+\infty]$. We will use the real interpolation spaces 
$$(D(T_2^k),X)_{\theta,q} = (X,D(T_2^k))_{1{\mbox{-}}\theta,q},$$ 
defined, for instance, in \cite{lions-peetre}, or in \cite{lunardi}. 

In particular, for $k=1$, we have the following characterization
$$(D(T_2),X)_{\theta,q} := \left\{\psi \in X:t\longmapsto t^{1{\mbox{-}}\theta} \|T_2(T_2-tI)^{{\mbox{-}}1}\psi\|_X \in L_*^q(0,+\infty)\right\},$$
where $L_*^q(0,+\infty)$ is given by
$$L_*^q(0,+\infty) := \left\{f \in L^q(0,+\infty) : \left(\int_0^{+\infty} \left\|f(t)\right\|^q\frac{dt}{t}\right)^{1/q} < + \infty \right\}, \quad \text{for } q \in [1,+\infty),$$
and for $q = +\infty$, by
$$L_*^\infty(0,+\infty;\CC) := \left\{f \text{ measurable on } (0,+\infty): \sup_{t \in (0,+\infty)} |f(t)| < + \infty\right\},$$
see \cite{da prato-grisvard} p. 325, or \cite{grisvard}, p. 665, Teorema 3, or section 1.14 of \cite{triebel}, where this space is denoted by $(X, D(T_2))_{1{\mbox{-}}\theta,q}$. Note that we can also characterize the space $(D(T_2),X)_{\theta,q}$ taking into account the Osservazione, p. 666, in \cite{grisvard}.

We set also, for any $k \in \NN\setminus\{0\}$
\begin{equation*}
(D(T_2),X)_{k+\theta,q}\;:=\;\left\{\psi\in D(T_2^k) : T_2^k\psi\in (D(T_2),X)_{\theta,q}\right\},
\end{equation*}
\begin{equation*}
(X,D(T_2))_{k+\theta,q}\;:=\;\left\{\psi\in D(T_2^k) : T_2^k\psi\in (X,D(T_2))_{\theta,q}\right\}.
\end{equation*}

\end{Def}
We recall the following lemma.
\begin{Lem}[\cite{grisvard}]\label{remtrace}
Let $T_2$ be a linear operator satisfying \eqref{def T2}. Let $u$ such that
$$u \in W^{n,p}(a_1,b_1;X) \cap L^p(a_1,b_1;D(T_2^k)),$$ 
where $a_1,b_1 \in \RR$ with $a_1<b_1$, $n,k \in \NN\setminus\{0\}$ and $p \in (1,+\infty)$. Then for any $j \in \NN$ satisfying the Poulsen condition $0<\frac{1}{p}+j < n$ and $s \in \{a_1,b_1\}$, we have
$$
u^{(j)}(s) \in (D(T_2^k),X)_{\frac{j}{n}+\frac{1}{np},p}.
$$
\end{Lem}
This result is proved in \cite{grisvard}, p. 678, Teorema 2'.

\subsection{Hypotheses}\label{sect hyp}

Througout this article, $k_+,k_- \in \RR_+ \setminus \{0\}$ and $A$ denotes a closed linear operator in $X$.
 
We assume the following hypotheses:
\begin{center}
\begin{tabular}{l}
$(H_1)\quad$ $X$ is a UMD space,\\ \ecart
$(H_2)\quad$ $0 \in \rho(A)$,\\ \ecart
$(H_3)\quad$ $-A \in$ BIP$(X,\theta_A)$ for some $\theta_A \in (0,\pi/2)$, \\ \ecart
$(H_4)\quad$ $-A \in$ Sect$(0)$. 
\end{tabular}
\end{center}
Note that assumption $(H_4)$ means that $-A$ is a sectorial operator of any angle $\theta \in (0, \pi)$. This assumption is satisfied, for instance, by elliptic differential operators of second order. Now, we give some consequences on our hypotheses.

\subsection{Consequences}\label{Consequences}

\begin{enumerate}
\item Note that $A_0$ satisfies all the previous hypotheses with $X=L^q(\omega)$, with $q \in (1, +\infty)$. From \cite{rubio}, Proposition~3, p. 207, $X$ satisfies $(H_1)$ and from \cite{gilbarg-trudinger}, Theorem~9.15, p.~241 and Lemma~9.17, p.~242, $A_0$ satisfies~$(H_2)$. Moreover, from \cite{pruss-sohr2}, Theorem~C, p.~166-167, $(H_3)$ is satisfied for every $\theta_A \in (0,\pi)$, thus $(H_4)$ is also satisfied.

\item In the scalar case, to solve each equation of $(EQ)$, it is necessary to introduce the square roots $\pm \sqrt{-A}$ of the characteristic equations 
$$x^4 + 2A x^2 + A^2 = 0,$$ 
this is why, in our operational case, we consider,
\begin{equation}\label{defM}
M\;:=\; -\sqrt{-A}.
\end{equation}
From $(H_3)$, $-A$ is a sectorial operator, thus the existence of $M$ is ensured, see for instance \cite{haase}, e), p. 25.

\item From $(H_3)$, we have $-A \in$ BIP$(X,\theta_A)$, then, from \cite{haase}, Proposition $3.2.1,$ e), p. 71, we deduce
$$-M \in\text{BIP}(X,\theta_A/2).$$ 
Since $0 < \theta_A/2 < \pi/2$, we deduce that $M$ generate a bounded analytic semigroup $(e^{xM})_{x \geqslant 0}$, see \cite{pruss-sohr}, Theorem $2$, p. 437. Furthermore, due to \cite{pruss-sohr}, Theorem $4$, p. 441, for $n \in \NN \setminus\{0\}$, we get $-nM \in \text{BIP}(X,\theta_A/4+\varepsilon)$, for any $\varepsilon \in (0,\pi/2-\theta_A/4)$. Then, due to \cite{pruss-sohr}, Theorem $2$, p. 437, $nM$ generate a bounded analytic semigroup $(e^{nxM})_{x \geqslant 0}$. The last results use also the works of \cite{da prato-grisvard} and \cite{dore-venni}.
\end{enumerate}

\subsection{The main results}

To solve our transmission problem (P), we introduce two auxiliary problems:
$$(P_-)\left\{\begin{array}{l}
u_-^{(4)}(x) + 2A u''_- (x) + A^2 u_-(x) = f_-(x), \quad \text{for a.e. } x \in (a, \gamma) \\ \ecart
\hspace{-0.15cm}\begin{array}{lllllll}
u_-(a) & = & \varphi_1^-, & u_-(\gamma) & = & \psi_1 \\ \ecart
u'_-(a) & = & \varphi_2^-, & u'_-(\gamma) & = & \psi_2,
\end{array}\end{array}\right.$$
and
$$(P_+)\left\{\begin{array}{l}
u_+^{(4)}(x) + 2A u''_+ (x) + A^2 u_+(x) = f_+(x), \quad \text{for a.e. } x \in (\gamma, b) \\ \ecart
\hspace{-0.15cm}\begin{array}{lllllll}
u_+(\gamma) & = & \psi_1, & u_+(b) & = & \varphi_1^+ \\ \ecart
u'_+(\gamma) & = & \psi_2, & u'_+(b) & = & \varphi_2^+.
\end{array}\end{array}\right.$$

\begin{Rem}\label{Rem P+ P-} 
Recall that a classical solution of $(P_\pm)$, in $L^p(J;X)$, with $J=(a,\gamma)$ or $(\gamma,b)$, is a solution to $(P_\pm)$ such that
$$u_\pm \in W^{4,p}(J;X)\cap L^p(J;D(A^2)), \quad u'' \in L^p(J;D(A)).$$
We say that $u$ is a classical solution of (P) if and only if there exist $\psi_1, \psi_2 \in X$ such that
\begin{center}
\begin{tabular}{l}
$~(i)\quad$ ~$u_-$ is a classical solution of $(P_-)$,\\ \ecart
$\,(ii)\quad$ \,$u_+$ is a classical solution of $(P_+)$,\\ \ecart
$(iii)\quad$ $u_-$ and $u_+$ satisfy $(TC2)$.
\end{tabular}
\end{center}
Note that by construction, if there exist a classical solution $u_-$ of $(P_-)$ and $u_+$ of $(P_+)$, then $u_-$ and $u_+$ satisfy $(TC1)$. 
\end{Rem}
Our goal is to prove that there exists a unique couple $(\psi_1,\psi_2)$ which satisfies $(i)$, $(ii)$ and $(iii)$. This will lead us to obtain our main result. 
\begin{Th} \label{Th principal}
Let $f_- \in L^p(a,\gamma; X)$ and $f_+ \in L^p(\gamma,b; X)$. Assume that $(H_1)$, $(H_2)$ and $(H_3)$ be satisfied. Then, there exists a unique classical solution $u$ of the transmission problem (P) if and only if 
\begin{equation}\label{reg phi +-}
\varphi_1^+, \varphi_1^- \in (D(A),X)_{1 + \frac{1}{2p},p} \quad  \text{and} \quad \varphi_2^+, \varphi_2^- \in (D(A),X)_{1+\frac{1}{2} + \frac{1}{2p},p}.
\end{equation}
\end{Th}

\begin{Rem}\hfill
\begin{enumerate}

\item In the proof of \refT{Th principal}, we use operator $M$ and interpolation spaces $(D(M),X)_{3-j+\frac{1}{p},p}$, $j=0,1,2,3$. But from the reiteration Theorem, we get
\begin{equation}\label{egalites espaces interpolation}
\hspace{-0.05cm}\left\{\hspace{-0.15cm}\begin{array}{lcllcl}
(D(M),X)_{3+\frac{1}{p},p} &=& (D(A),X)_{1+\frac{1}{2p},p}, & (D(M),X)_{2+\frac{1}{p},p} &=& (D(A),X)_{1+\frac{1}{2}+\frac{1}{2p},p} \\ \ecart
(D(M),X)_{1+\frac{1}{p},p} &=& (D(A),X)_{\frac{1}{2p},p}, & (D(M),X)_{\frac{1}{p},p} &=& (D(A),X)_{\frac{1}{2}+\frac{1}{2p},p}.
\end{array}\right.
\end{equation}

\item We can generalize the previous Theorem by considering a transmission problem between $N$ juxtaposed habitats, with $N \in \NN \setminus \{0\}$. For instance, with $N=3$, it suffices to use \refT{Th principal} on the two first habitats and then apply it on the transmission problem between the second and third habitat to solve the problem. By recurrence, we obtain the result for $N$ habitats.
\end{enumerate}
\end{Rem}
As consequence of \refT{Th principal}, we deduce the following result for problem $(P_{pde})$. Consider the case $A=A_0$ (other cases can be treated).

\begin{Cor}\label{CorAppli}
Assume that $\omega$ is a bounded open set of $\RR^{n-1}$ where $n \geq 2$ with $C^2$-boundary. Let $g_+\in L^p(\Omega_+)$ and $g_-\in L^p(\Omega_-)$ with $p\in(1,+\infty)$ and $\dis p > n$; let $k_+,k_- \in \RR_+\setminus\{0\}$. Then, there exists a unique solution $u$ of $(P_{pde})$, such that 
$$u_-\in W^{4,p}(\Omega_-), \quad u_+\in W^{4,p}(\Omega_+),$$
if and only if 
$$
\varphi_1^\pm, \varphi_2^\pm \in  W^{2,p}(\omega)\cap W^{1,p}_0(\omega),~~\Delta \varphi_1^\pm, \in W^{2-\frac{1}{p},p}(\omega) \cap W_0^{1,p}(\omega)~~\text{and}~~ \Delta\varphi_2^\pm \in W^{1-\frac{1}{p},p}(\omega) \cap W_0^{1,p}(\omega).
$$
\end{Cor}
\begin{proof}
The proof is quite similar to the one of Corollary~3.6 in \cite{LLMT}, see also Corollary 2.7 in \cite{LMMT}.
\end{proof}
Taking into account the result of \refT{Th principal}, we can also obtain anisotropic results by considering $f_- \in L^p(a,\gamma;L^q(\omega))$ and $f_+ \in L^p(\gamma,b;L^q(\omega))$ with $p,q \in (1,+\infty)$.

\section{Preliminary results}\label{Section Proof prel res}

Throughout this article, we set 
$$c = \gamma-a > 0\quad \text{and} \quad d = b-\gamma > 0.$$ 
From \refR{Rem P+ P-}, to solve problem (P), we first have to study problems $(P_-)$ and $(P_+)$. To this end, we need the following invertibility result obtained in \cite{thorel}. 
\begin{Lem}[\cite{thorel}]\label{lem U+- V+- inv}
The operators $U_+$, $U_-$, $V_+$, $V_- \in \L(X)$ defined by
\begin{equation}\label{U V +-}
\left \{\begin{array}{rclrcl}
U_- & := & \dis I - e^{2c M} + 2c Me^{c M}, & U_+ & := & \dis I - e^{2d M} + 2d Me^{d M}\\ \ecart
V_- & := & \dis I - e^{2c M} - 2c Me^{c M}, & V_+ & := & \dis I - e^{2d M} - 2d Me^{d M},
\end{array}\right.
\end{equation}
are invertible with bounded inverse.
\end{Lem}
All these exponentials are analytic semigroups and they are well defined due to statement $3$ of Section \ref{Consequences}. For a detailed proof, see~\cite{thorel}, Proposition 4.5. 

\subsection{Problem $\mathbf{(P_-)}$}

\begin{Prop}\label{prop sol à gauche M}
Let $f_- \in L^p(a,\gamma; X)$. Assume that $(H_1)$, $(H_2)$ and $(H_3)$ hold. Then there exists a unique classical solution $u_-$ of problem $(P_-)$ if and only if 
\begin{equation} \label{reg phi/psi i P- M}
\varphi_1^-,\psi_1 \in (D(A),X)_{1+\frac{1}{2p},p} \quad  \text{and} \quad \varphi_2^-,\psi_2 \in (D(A),X)_{1+\frac{1}{2}+\frac{1}{2p},p}.
\end{equation}
Moreover 
\begin{equation}\label{rep u- M}
\begin{array}{rcl}
u_-(x) &\hspace*{-0.1cm} = &\hspace*{-0.1cm} \dis ~~~\left(e^{(x-a) M} - e^{(\gamma-x) M}\right)\alpha_1^- + \left((x-a)e^{(x-a) M} - (\gamma-x)e^{(\gamma-x)M} \right) \alpha_2^- \\ \ecart
&\hspace*{-0.1cm}&\hspace*{-0.1cm} \dis + \left(e^{(x-a) M} + e^{(\gamma-x) M}\right)\alpha_3^- + \left((x-a)e^{(x-a)M} + (\gamma-x)e^{(\gamma-x)M}\right) \alpha_4^- + F_-(x),
\end{array}
\end{equation}
where
\begin{equation}\label{ali- M}
\left \{ \begin{array}{rcr}
\alpha_1^- & := & \dis -\frac{1}{2} U_-^{-1} \left(\left(I + \left( I + c M \right) e^{cM} \right)\psi_1 - ce^{cM} \psi_2\right) + \tilde{\varphi}_1^- \\ \ecart

\alpha_2^- & := & \dis \frac{1}{2} U_-^{-1} \left( \left(I + e^{cM}\right) M \psi_1 + \left(I - e^{cM}\right) \psi_2 \right) + \tilde{\varphi}_2^- \\ \ecart

\alpha_3^- & := & \dis \frac{1}{2} V_-^{-1} \left(\left(I - \left( I + c M \right) e^{cM} \right) \psi_1 + c e^{cM} \psi_2\right) + \tilde{\varphi}_3^- \\ \ecart

\alpha_4^- & := & \dis - \frac{1}{2} V_-^{-1} \left( \left(I - e^{cM}\right) M \psi_1 + \left(I + e^{cM}\right) \psi_2 \right) + \tilde{\varphi}_4^-,
\end{array}\right.
\end{equation}
with
\begin{equation}\label{phi i tilde - M}
\left\{\begin{array}{rcl}
\tilde{\varphi_1}^- & := & \dis \frac{1}{2} U_-^{-1} \left(\varphi_1^- + e^{cM} \left(\varphi_1^- + c \left( M \varphi_1^- + \varphi_2^- - F_-'(a) - F_-'(\gamma)\right) \right)\right) \\ \\

\tilde{\varphi_2}^- & := & \dis -\frac{1}{2} U_-^{-1} \left(M \varphi_1^- - \varphi_2^- + F_-'(a) + F_-'(\gamma)\right) \\ \ecart
&& \dis -\frac{1}{2} U_-^{-1} e^{cM} \left( M \varphi_1^- + \varphi_2^- - F_-'(a) - F_-'(\gamma) \right)  \\ \\

\tilde{\varphi_3}^- & := & \dis \frac{1}{2} V_-^{-1} \left(\varphi_1^- - e^{cM} \left(\varphi_1^- + c \left( M \varphi_1^- + \varphi_2^- - F_-'(a) + F_-'(\gamma)\right) \right)\right) \\ \\
 
\tilde{\varphi_4}^- & := & \dis -\frac{1}{2} V_-^{-1} \left(M \varphi_1^- - \varphi_2^- + F_-'(a) - F_-'(\gamma)\right) \\ \ecart
 && \dis + \frac{1}{2} V_-^{-1} e^{cM} \left( M \varphi_1^- + \varphi_2^- - F_-'(a) + F_-'(\gamma) \right),
\end{array}\right.
\end{equation} 
and $F_-$ is the unique classical solution of problem
\begin{equation}\label{pb F- M} \left\{\begin{array}{l}
u_-^{(4)}(x) + 2A u''_- (x) + A^2 u_-(x) = f_-(x), \quad \text{a.e. } x \in (a, \gamma) \\ \ecart
u_-(a) = u_-(\gamma) = 0 \\ \ecart
u''_-(a) = u''_-(\gamma) = 0.
\end{array}\right.
\end{equation}
\end{Prop}

\begin{proof} 
Due to \cite{thorel}, Theorem 2.8, statement 2., there exists a unique classical solution $u_-$ of problem $(P_-)$ if and only if \eqref{reg phi/psi i P- M} holds. 

Moreover, adapting the representation formula of $u$ given by $(31)$ in \cite{thorel}, where $u$, $f$, $b$, $F_{0,f}$, $\varphi_1$, $\varphi_2$, $\varphi_3$ and $\varphi_4$ are respectively replaced by $u_-$, $f_-$, $\gamma$, $F_-$, $\varphi_1^-$, $\psi_1$, $\varphi_2^-$ and $\psi_2$, we obtain that the representation formula of $u_-$ is given by \eqref{rep u- M}, \eqref{ali- M} and \eqref{phi i tilde - M}. 
\end{proof}

\begin{Rem}\label{Rem al- M}
In the previous proposition, since \eqref{reg phi/psi i P- M}, \eqref{ali- M} and \eqref{phi i tilde - M} hold, we have 
$$\alpha^-_1,\alpha^-_3 \in D(M^3) \quad \text{and} \quad \alpha^-_2,\alpha^-_4 \in D(M^2).$$
Moreover, since $F_-$ is a classical solution of \eqref{pb F- M}, due to \refL{remtrace}, for $j = 0,1,2,3$ and $s=a$ or $\gamma$
$$F_-^{(j)}(s) \in (D(M),X)_{3-j + \frac{1}{p},p}.$$
\end{Rem}

\subsection{Problem $\mathbf{(P_+)}$}

\begin{Prop}\label{prop sol à droite M}
Let $f_+ \in L^p(\gamma,b; X)$. Assume that $(H_1)$, $(H_2)$ and $(H_3)$ hold. Then, there exists a unique classical solution $u_+$ of problem $(P_+)$ if and only if 
\begin{equation}\label{reg phi/psi i P+ M}
\varphi_1^+,\psi_1 \in (D(A),X)_{1+\frac{1}{2p},p} \quad  \text{and} \quad \varphi_2^+,\psi_2 \in (D(A),X)_{1+\frac{1}{2}+\frac{1}{2p},p}.
\end{equation}
Moreover
\begin{equation}\label{rep u+ M}
\begin{array}{rcl}
u_+(x) & = & \dis ~~~\left(e^{(x-\gamma) M} - e^{(b-x) M}\right)\alpha_1^+ + \left((x-\gamma)e^{(x-\gamma) M} - (b-x)e^{(b-x)M} \right) \alpha_2^+ \\ \ecart
&& \dis + \left(e^{(x-\gamma) M} + e^{(b-x) M}\right)\alpha_3^+ + \left((x-\gamma)e^{(x-\gamma)M} + (b-x)e^{(b-x)M}\right) \alpha_4^+ \\ \ecart
&& \dis + F_+(x),
\end{array}
\end{equation}
where
\begin{equation}\label{ali+ M}
\left \{ \begin{array}{rcr}
\alpha_1^+ & := & \dis  \frac{1}{2} U_+^{-1} \left(\left(I + \left( I + d M \right) e^{dM} \right) \psi_1 + d e^{dM} \psi_2 \right) + \tilde{\varphi}_1^+\\ \\

\alpha_2^+ & := & \dis -\frac{1}{2} U_+^{-1} \left( \left(I + e^{dM}\right) M \psi_1 - \left(I - e^{dM}\right) \psi_2 \right) + \tilde{\varphi}_2^+ \\ \\

\alpha_3^+ & := & \dis \frac{1}{2} V_+^{-1} \left(\left(I - \left( I + d M \right) e^{dM} \right) \psi_1 - d e^{dM} \psi_2 \right) + \tilde{\varphi}_3^+ \\ \\

\alpha_4^+ & := & \dis - \frac{1}{2} V_+^{-1} \left( \left(I - e^{dM}\right) M \psi_1 - \left(I + e^{dM}\right) \psi_2 \right) + \tilde{\varphi}_4^+,
\end{array}\right.
\end{equation}
with
\begin{equation}\label{phi i tilde + M}
\left\{\begin{array}{rcl}
\tilde{\varphi_1}^+ & := & \dis - \frac{1}{2} U_+^{-1} \left(\varphi_1^+ + e^{dM} \left(\varphi_1^+ + d \left(M \varphi_1^+ - \varphi_2^+ + F_+'(\gamma) + F_+'(b) \right) \right) \right) \\ \\

\tilde{\varphi_2}^+ & := & \dis \frac{1}{2} U_+^{-1} \left( M \varphi_1^+ + \varphi_2^+ - F_+'(\gamma) - F_+'(b) \right) \\ \ecart
&& \dis + \frac{1}{2} U_+^{-1} e^{dM} \left( M\varphi_1^+ - \varphi_2^+ + F_+'(\gamma) + F_+'(b) \right) \\ \\

\tilde{\varphi_3}^+ & := & \dis \frac{1}{2} V_+^{-1} \left(\varphi_1^+ - e^{dM} \left( \varphi_1^+ + d  \left(M \varphi_1^+ - \varphi_2^+ - F_+'(\gamma) + F_+'(b)\right) \right) \right) \\ \\

\tilde{\varphi_4}^+ & := & \dis - \frac{1}{2} V_+^{-1} \left( M \varphi_1^+ + \varphi_2^+ + F_+'(\gamma) - F_+'(b) \right) \\ \ecart
&& \dis + \frac{1}{2} V_+^{-1} e^{dM} \left( M \varphi_1^+ - \varphi_2^+ - F_+'(\gamma) + F_+'(b) \right),
\end{array}\right.
\end{equation} 
and $F_+$ is the unique classical solution of problem
\begin{equation}\label{pb F+ M}
\left\{\begin{array}{l}
u_+^{(4)}(x) + 2A u''_+ (x) + A^2 u_+(x) = f_+(x), \quad \text{a.e. } x \in (\gamma, b) \\ \ecart
u_+(\gamma) = u_+(b) = u''_+(\gamma) = u''_+(b) = 0.
\end{array}\right.
\end{equation}
\end{Prop}

\begin{proof}
Due to \cite{thorel}, Theorem 2.8, statement 2., there exists a unique classical solution $u_+$ of problem $(P_+)$ if and only if \eqref{reg phi/psi i P+ M} holds. Moreover, adapting the representation formula of $u$ given by $(31)$ in \cite{thorel}, where $u$, $f$, $a$, $F_{0,f}$, $\varphi_1$, $\varphi_2$, $\varphi_3$ and $\varphi_4$ are respectively replaced by $u_+$, $f_+$, $\gamma$, $F_+$, $\psi_1$, $\varphi_1^+$, $\psi_2$ and $\varphi_2^+$, we obtain that the representation formula of $u_+$ is given by \eqref{rep u+ M}, \eqref{ali+ M} and \eqref{phi i tilde + M}. 
\end{proof}

\begin{Rem}\label{Rem al+ M}
In the previous proposition, since \eqref{reg phi/psi i P+ M}, \eqref{ali+ M} and \eqref{phi i tilde + M} hold, we have 
$$\alpha^+_1,\alpha^+_3 \in D(M^3) \quad \text{and} \quad \alpha^+_2,\alpha^+_4 \in D(M^2).$$
Moreover, since $F_+$ is a classical solution of \eqref{pb F+ M}, due to \refL{remtrace}, for $j = 0,1,2,3$ and $s=\gamma$ or $b$
$$F_+^{(j)}(s) \in (D(M),X)_{3-j + \frac{1}{p},p}.$$
\end{Rem}

\subsection{The transmission system}

In this section we give the proof of \refT{Th syst trans M} stated below. This theorem ensures the equivalence between the resolution of problem $(P)$ and the resolution of the following system
\begin{equation}\label{syst trans M}
\left\{\begin{array}{lcl}
\left(P_1^+ + P_1^-\right) M \psi_1 - \left(P_2^+ - P_2^-\right) \psi_2 & = & S_1 \\ \ecart

\left(P_2^+ - P_2^-\right) M \psi_1 - \left(P_3^+ + P_3^-\right) \psi_2 & = & S_2.
\end{array}\right.
\end{equation} 
The coefficients of the previous system are given by 
\begin{equation}\label{P1+ P2+ P3+ M}
\left\{\begin{array}{lcl}
P_1^+ & = & \dis k_+ \left(U_+^{-1} \left(I + e^{dM}\right)^2 + V_+^{-1} \left(I - e^{dM}\right)^2\right) \\ \ecart

P_2^+ & = & \dis k_+ \left(U_+^{-1} + V_+^{-1}\right) \left(I - e^{2dM}\right) \\ \ecart

P_3^+ & = & \dis k_+ \left(U_+^{-1}\left(I - e^{dM}\right)^2 + V_+^{-1} \left(I + e^{dM}\right)^2 \right)
\end{array}\right.
\end{equation} 
and
\begin{equation}\label{P1- P2- P3- M}
\left\{\begin{array}{lcl}
P_1^- & = & k_- \left(U_-^{-1} \left(I + e^{cM}\right)^2 + V_-^{-1} \left(I - e^{cM}\right)^2\right) \\ \ecart

P_2^- & = & k_- \left(U_-^{-1} + V_-^{-1} \right)  \left(I - e^{2cM}\right) \\ \ecart

P_3^- & = & k_- \left(U_-^{-1}\left(I - e^{cM}\right)^2 + V_-^{-1} \left(I + e^{cM}\right)^2 \right).
\end{array}\right.
\end{equation}
The seconds members are given by  
\begin{equation}\label{S1 M}
S_1 = 2 k_+ \left(\left(\tilde{\varphi}_2^+ + \tilde{\varphi}_4^+ \right) + e^{dM} \left( \tilde{\varphi}_2^+ - \tilde{\varphi}_4^+\right)\right) - 2 k_- \left(\left(\tilde{\varphi}_2^- - \tilde{\varphi}_4^-\right)  + e^{c M} \left(\tilde{\varphi}_2^- + \tilde{\varphi}_4^-\right)\right) - M^{-2}\check{S}
\end{equation}
where
\begin{equation}\label{R M}
\check{S} = -k_+ F^{(3)}_+(\gamma) + k_+ M^2 F'_+(\gamma) + k_- F^{(3)}_-(\gamma) - k_- M^2 F'_-(\gamma)
\end{equation}
and
\begin{equation}\label{S2 M}
S_2 = 2 k_+ \left(\left(\tilde{\varphi}_2^+ + \tilde{\varphi}_4^+ \right) - e^{dM} \left( \tilde{\varphi}_2^+ - \tilde{\varphi}_4^+\right)\right) + 2 k_- \left(\left(\tilde{\varphi}_2^- - \tilde{\varphi}_4^-\right) - e^{c M} \left(\tilde{\varphi}_2^- + \tilde{\varphi}_4^-\right)\right).
\end{equation}
\begin{Th}\label{Th syst trans M}
Let $f_- \in L^p(a,\gamma;X)$ and $f_+ \in L^p(\gamma,b;X)$. Assume that $(H_1)$, $(H_2)$ and $(H_3)$ hold. Then, problem (P) has a unique classical solution if and only if the data $\varphi_1^+$, $\varphi_1^-$, $\varphi^+_2$, $\varphi^-_2$ satisfy \eqref{reg phi +-} and system \eqref{syst trans M} has a unique solution $(\psi_1,\psi_2)$ such that 
\begin{equation}\label{reg (psi1,psi2) M}
(\psi_1,\psi_2) \in (D(A),X)_{1+\frac{1}{2p},p} \times (D(A),X)_{1+\frac{1}{2}+\frac{1}{2p},p}.
\end{equation}
\end{Th}

\begin{proof}
Assume that problem $(P)$ has a unique classical solution $u$. We set
$$\psi_1 = u_-(\gamma) = u_+(\gamma) \quad \text{and} \quad \psi_2 = u'_-(\gamma) = u'_+(\gamma).$$
We get that $u_-$ (respectively $u_+$) is the classical solution of $(P_-)$ (respectively $(P_+)$). Then, applying \refP{prop sol à gauche M} (respectively \refP{prop sol à droite M}), we obtain \eqref{reg phi +-}. Moreover, from \eqref{egalites espaces interpolation}, we have
$$\psi_1 \in (D(M),X)_{3+\frac{1}{p},p} \quad  \text{and} \quad \psi_2 \in (D(M),X)_{2 + \frac{1}{p},p}.$$
It remains to prove that $(\psi_1,\psi_2)$ is solution of system \eqref{syst trans M}. To this end, since $u$ satisfies the transmission conditions $(TC2)$. Then, we obtain the following system
\begin{equation*}
\left\{\begin{array}{lllllll}
k_- u_-^{(3)}(\gamma) & + & k_- Au'_-(\gamma) & = & k_+ u_+^{(3)}(\gamma) & + & k_+ Au'_+(\gamma) \\ \ecart

k_- u''_-(\gamma) & + & k_- Au_-(\gamma) & = & k_+ u''_+(\gamma) & + & k_+ Au_+(\gamma).
\end{array}\right.
\end{equation*} 
Hence
$$\left\{\begin{array}{lllll}
k_+ \left(u_+^{(3)}(\gamma) - M^2 u'_+(\gamma)\right) & - & k_- \left(u_-^{(3)}(\gamma) - M^2 u'_-(\gamma)\right) & = & 0 \\ \ecart

k_+ \left(u''_+(\gamma) - M^2 u_+(\gamma) \right) & - & k_- \left(u''_-(\gamma) - M^2 u_-(\gamma)\right) & = & 0. 
\end{array}\right.$$
Moreover, for all $x \in (a,\gamma)$, we have
$$\begin{array}{rcl}
u_-(x) & = & \dis \left(e^{(x-a) M} - e^{(\gamma-x) M}\right)\alpha_1^- + \left((x-a)e^{(x-a) M} - (\gamma-x)e^{(\gamma-x)M} \right) \alpha_2^- \\ \ecart
&& \dis + \left(e^{(x-a) M} + e^{(\gamma-x) M}\right)\alpha_3^- + \left((x-a)e^{(x-a)M} + (\gamma-x)e^{(\gamma-x)M}\right) \alpha_4^- \\ \ecart
&& + F_-(x) \\ \\

u'_-(x) & = & \dis M \left(e^{(x-a) M} + e^{(\gamma-x) M}\right)\alpha_1^- + M \left(e^{(x-a) M} - e^{(\gamma-x) M}\right)\alpha_3^-  \\ \ecart
&& \dis + \left((I + (x-a)M)e^{(x-a) M} + (I + (\gamma-x)M) e^{(\gamma-x)M} \right) \alpha_2^- \\ \ecart
&& \dis  + \left((I + (x-a)M) e^{(x-a)M} - (I + (\gamma-x)M) e^{(\gamma-x)M}\right) \alpha_4^- + F'_-(x )\\ \\

u''_-(x) & = & \dis M^2 \left(e^{(x-a) M} - e^{(\gamma-x) M}\right)\alpha_1^- + M^2 \left(e^{(x-a) M} + e^{(\gamma-x) M}\right)\alpha_3^- + F''_-(x) \\ \ecart
&& \dis + \left((2M + (x-a)M^2) e^{(x-a) M} - (2M + (\gamma-x)M^2) e^{(\gamma-x)M} \right) \alpha_2^- \\ \ecart
&& \dis + \left((2M + (x-a)M^2)e^{(x-a)M} + (2M + (\gamma-x)M^2)e^{(\gamma-x)M}\right) \alpha_4^-, \\ \\ 

u^{(3)}_-(x) & = & \dis M^3 \left(e^{(x-a) M} + e^{(\gamma-x) M}\right)\alpha_1^- + M^3 \left(e^{(x-a) M} - e^{(\gamma-x) M}\right)\alpha_3^- + F^{(3)}_-(x) \\ \ecart
&& \dis + \left((3M^2 + (x-a)M^3)e^{(x-a) M} + (3M^2 + (\gamma-x)M^3) e^{(\gamma-x)M} \right) \alpha_2^- \\ \ecart
&& \dis + \left((3M^2 + (x-a)M^3) e^{(x-a)M} - (3M^2 + (\gamma-x)M^3) e^{(\gamma-x)M}\right) \alpha_4^-.
\end{array}$$
Thus, we obtain
\begin{equation}\label{u'''-u' -}
\begin{array}{lllll}
\dis k_-\left(u^{(3)}_-(\gamma) - M^2 u'_-(\gamma)\right) &=& \dis k_- \left(2M^2\left(I + e^{c M}\right) \alpha_2^- - 2 M^2 \left(I - e^{c M}\right) \alpha_4^-\right) \\ \ecart
&& \dis + k_- F^{(3)}_-(\gamma) - k_- M^2 F'_-(\gamma).
\end{array}
\end{equation}
Moreover, from \eqref{pb F- M}, we have
\begin{equation}\label{u''-u -}
k_-\left(u''_-(\gamma) - M^2 u_-(\gamma)\right) = -k_-\left(2M\left(I - e^{c M} \right) \alpha_2^- - 2M\left(I + e^{c M} \right) \alpha_4^-\right).
\end{equation}
Note that, from \refR{Rem al- M}, all terms of the previous equality are well defined. 

In the same way, for all $x \in (\gamma,b)$, we have
$$\begin{array}{rcl}
u_+(x) & = & \dis \left(e^{(x-\gamma) M} - e^{(b-x) M}\right)\alpha_1^+ + \left((x-\gamma)e^{(x-\gamma) M} - (b-x)e^{(b-x)M} \right) \alpha_2^+ \\ \ecart
&& \dis + \left(e^{(x-\gamma) M} + e^{(b-x) M}\right)\alpha_3^+ + \left((x-\gamma)e^{(x-\gamma)M} + (b-x)e^{(b-x)M}\right) \alpha_4^+ \\ \ecart
&& + F_+(x), \\ \\

u'_+(x) & = & \dis M \left(e^{(x-\gamma) M} + e^{(b-x) M}\right)\alpha_1^+ + M \left(e^{(x-\gamma) M} - e^{(b-x) M}\right)\alpha_3^+ + F'_+(x) \\ \ecart
&& \dis + \left((I + (x-\gamma)M)e^{(x-\gamma) M} + (I + (b-x)M) e^{(b-x)M} \right) \alpha_2^+ \\ \ecart
&& \dis  + \left((I + (x-\gamma)M) e^{(x-\gamma)M} - (I + (b-x)M) e^{(b-x)M}\right) \alpha_4^+, 
\end{array}$$
and
$$\begin{array}{rcl}
u''_+(x) & = & \dis M^2 \left(e^{(x-\gamma) M} - e^{(b-x) M}\right)\alpha_1^+ + M^2 \left(e^{(x-\gamma) M} + e^{(b-x) M}\right)\alpha_3^+  \\ \ecart
&& \dis + \left((2M + (x-\gamma)M^2) e^{(x-\gamma) M} - (2M + (b-x)M^2) e^{(b-x)M} \right) \alpha_2^+ \\ \ecart
&& \dis + \left((2M + (x-\gamma)M^2)e^{(x-\gamma)M} + (2M + (b-x)M^2)e^{(b-x)M}\right) \alpha_4^+ \\ \ecart
&& + F''_+(x), \\ \\

u^{(3)}_+(x) & = & \dis M^3 \left(e^{(x-\gamma) M} + e^{(b-x) M}\right)\alpha_1^+ + M^3 \left(e^{(x-\gamma) M} - e^{(b-x) M}\right)\alpha_3^+  \\ \ecart
&& \dis + \left((3M^2 + (x-\gamma)M^3)e^{(x-\gamma) M} + (3M^2 + (b-x)M^3) e^{(b-x)M} \right) \alpha_2^+ \\ \ecart
&& \dis + \left((3M^2 + (x-\gamma)M^3) e^{(x-\gamma)M} - (3M^2 + (b-x)M^3) e^{(b-x)M}\right) \alpha_4^+ \\ \ecart
&& + F^{(3)}_+(x).
\end{array}$$
Hence, we obtain
\begin{equation}\label{u'''-u' +}
\begin{array}{lllll}
\dis k_+\left(u^{(3)}_+(\gamma) - M^2 u'_+(\gamma)\right) &=& k_+ \left(2M^2 \left(I + e^{d M} \right) \alpha_2^+ + 2M^2 \left(I - e^{d M} \right) \alpha_4^+ \right) \\ \ecart
&& \dis + k_+ F^{(3)}_+(\gamma) - k_+ M^2 F'_+(\gamma). 
\end{array}
\end{equation}
Moreover, from \eqref{pb F+ M}, we have
\begin{equation}\label{u''-u +}
k_+\left(u''_+(\gamma) - M^2 u_+(\gamma)\right) = k_+\left(2M\left(I - e^{d M} \right) \alpha_2^+ + 2M \left(I + e^{d M} \right) \alpha_4^+\right).
\end{equation}
As previously, from \refR{Rem al+ M}, all terms of the previous equality are justified. Then, from \eqref{R M}, \eqref{u'''-u' -}, \eqref{u''-u -}, \eqref{u'''-u' +} and \eqref{u''-u +}, we deduce that system \eqref{syst trans M} is equivalent to
$$\left\{\begin{array}{lllll}
\dis k_+ \left(2M^2 \left(I + e^{d M} \right) \alpha_2^+ + 2M^2 \left(I - e^{d M} \right) \alpha_4^+ \right) \\ \ecart
- k_- \left(2M^2\left(I + e^{c M}\right) \alpha_2^- - 2 M^2 \left(I - e^{c M}\right) \alpha_4^-\right)  &=& \check{S} \\ \\

\dis k_+\left(2M\left(I - e^{d M} \right) \alpha_2^+ + 2M \left(I + e^{d M} \right) \alpha_4^+\right) \\ \ecart
+ k_-\left(2M\left(I - e^{c M} \right) \alpha_2^- - 2M\left(I + e^{c M} \right) \alpha_4^-\right) &=& 0.
\end{array}\right.$$
Thus, we obtain the following system
\begin{equation}\label{syst trans 2 M}
\left\{\begin{array}{lllll}
\dis k_+ \left(\left(I + e^{d M} \right) \alpha_2^+ + \left(I - e^{d M} \right) \alpha_4^+ \right) \\ \ecart
- k_- \left(\left(I + e^{c M}\right) \alpha_2^- - \left(I - e^{c M}\right) \alpha_4^-\right)  &=& \dis\frac{1}{2}M^{-2} \check{S} \\ \\

\dis k_+\left(\left(I - e^{d M} \right) \alpha_2^+ + \left(I + e^{d M} \right) \alpha_4^+\right) \\ \ecart
+ k_-\left(\left(I - e^{c M} \right) \alpha_2^- - \left(I + e^{c M} \right) \alpha_4^-\right) &=& 0.
\end{array}\right.
\end{equation}
Now, we regroup all source terms provided from \eqref{ali- M}, \eqref{phi i tilde - M}, \eqref{ali+ M} and \eqref{phi i tilde + M} in $S_1$ and $S_2$, defined by \eqref{S1 M}, \eqref{R M} and \eqref{S2 M}. Then, we deduce that system \eqref{syst trans 2 M} writes 
$$
\left\{\begin{array}{lll}
\dis -\frac{k_+}{2} U_+^{-1}\left(I + e^{dM}\right)\left( \left(I + e^{dM}\right) M \psi_1 - \left(I - e^{dM}\right) \psi_2 \right) \\ \ecart 

\dis - \frac{k_+}{2} V_+^{-1} \left(I - e^{dM}\right)\left( \left(I - e^{dM}\right) M \psi_1 - \left(I + e^{dM}\right) \psi_2 \right) \\ \ecart

\dis -\frac{k_-}{2} U_-^{-1} \left(I + e^{cM}\right)\left( \left(I + e^{cM}\right) M \psi_1 + \left(I - e^{cM}\right) \psi_2 \right) \\ \ecart 

\dis - \frac{k_-}{2} V_-^{-1} \left(I - e^{cM}\right)\left( \left(I - e^{cM}\right) M \psi_1 + \left(I + e^{cM}\right) \psi_2 \right) & = & \dis -\frac{S_1}{2} \\ \\

\dis -\frac{k_+}{2} U_+^{-1} \left(I - e^{dM}\right)\left( \left(I + e^{dM}\right) M \psi_1 - \left(I - e^{dM}\right) \psi_2 \right) \\ \ecart

\dis - \frac{k_+}{2} V_+^{-1} \left(I + e^{dM}\right)\left( \left(I - e^{dM}\right) M \psi_1 - \left(I + e^{dM}\right) \psi_2 \right) \\ \ecart

\dis + \frac{k_-}{2} U_-^{-1}\left(I - e^{cM}\right) \left( \left(I + e^{cM}\right) M \psi_1 + \left(I - e^{cM}\right) \psi_2 \right) \\ \ecart 

\dis + \frac{k_-}{2} V_-^{-1} \left(I + e^{cM}\right) \left( \left(I - e^{cM}\right) M \psi_1 + \left(I + e^{cM}\right) \psi_2 \right) & = & \dis - \frac{S_2}{2}.
\end{array}\right.
$$
Finally, we obtain
\begin{equation}\label{syst trans 3 M}
\left\{\begin{array}{rll}
\dis k_+ \left(U_+^{-1} \left(I + e^{dM}\right)^2 + V_+^{-1} \left(I - e^{dM}\right)^2\right) M \psi_1  \\ \ecart 

\dis + k_- \left(U_-^{-1} \left(I + e^{cM}\right)^2 + V_-^{-1} \left(I - e^{cM}\right)^2\right)M \psi_1 \\ \ecart 

\dis - k_+ \left(U_+^{-1} + V_+^{-1}\right) \left(I - e^{2dM}\right) \psi_2 \\ \ecart

\dis + k_- \left(U_-^{-1} + V_-^{-1} \right)  \left(I - e^{2cM}\right) \psi_2 & = & S_1 \\ \\

\dis k_+ \left(U_+^{-1} + V_+^{-1} \right) \left(I - e^{2dM}\right) M \psi_1 \\ \ecart

\dis - k_- \left( U_-^{-1} + V_-^{-1} \right) \left(I - e^{2cM}\right) M \psi_1 \\ \ecart 

\dis - k_+ \left(U_+^{-1}\left(I - e^{dM}\right)^2 + V_+^{-1} \left(I + e^{dM}\right)^2 \right) \psi_2 \\ \ecart

\dis - k_- \left(U_-^{-1}\left(I - e^{cM}\right)^2 + V_-^{-1} \left(I + e^{cM}\right)^2 \right) \psi_2 & = & S_2.
\end{array}\right.
\end{equation}
Then, using \eqref{P1+ P2+ P3+ M} and \eqref{P1- P2- P3- M} system \eqref{syst trans 3 M} become system \eqref{syst trans M}. So, $(\psi_1,\psi_2)$ is solution of system \eqref{syst trans M}.

Conversely, if \eqref{reg phi +-} holds and system \eqref{syst trans M} has a unique solution $(\psi_1,\psi_2)$ such that
$$\psi_1 \in (D(A),X)_{1+\frac{1}{2p},p} \quad \text{and} \quad \psi_2 \in (D(A),X)_{1+\frac{1}{2}+\frac{1}{2p},p},$$ 
Then, considering $u_-$ (respectively $u_+$) the unique classical solution of problem $(P_-)$ (respectively problem $(P_+)$), we obtain that 
$$u=\left\{ 
\begin{array}{l}
u_- \text{ \ \ in \ }\Omega_- \\ 
u_+ \text{ \ \ in \ }\Omega_+,
\end{array}\right.$$ 
is the unique classical solution of problem $(P)$.
\end{proof}

\subsection{Functional calculus} \label{sect funct calc}

To prove \refT{Th principal}, it remains, from \refT{Th syst trans M}, to solve system \eqref{syst trans M}. To this end, we have to show that the determinant operator of system \eqref{syst trans M} is an invertible operator using functional calculus.

We first recall some classical notations. Let $\theta \in (0,\pi)$, we denote by $H(S_\theta)$ the space of holomorphic functions on $S_\theta$ (defined by \eqref{defsector}). Moreover, we consider the following subspace of~$H(S_\theta)$:
$$
\mathcal{E}_\infty(S_\theta)\;:=\;\left\{f \in H(S_\theta) : f = O(|z|^{-s}) ~(|z| \rightarrow +\infty) \text{ for some } s > 0\right\}.
$$
Thus, $\mathcal{E}_\infty(S_\theta)$ is the space of polynomially decreasing holomorphic functions at $+\infty$. 

Let $T$ be an invertible sectorial operator of angle $\theta_T \in (0,\pi)$ and let $f\in \mathcal{E}_\infty(S_\theta)$, with $\theta \in (\theta_T,\pi)$, then, by functional calculus, we can define $f(T) \in \L(X)$, see \cite{haase}, p. 45.

We recall a useful invertibility result from \cite{LMMT}, Lemma 5.3.

\begin{Lem}[\cite{LMMT}]\label{lem inv}
Let $P$ be an invertible sectorial operator in $X$ with angle $\theta$, for all $\theta \in (0, \pi)$. Let \mbox{$G \in H(S_{\theta})$}, for some $\theta \in (0, \pi)$, such that
\begin{enumerate}
\item[$(i)$] $1-G \in \mathcal{E}_\infty(S_{\theta})$,

\item[$(ii)$] $G (x)\neq 0$ for any $x \in \RR_+ \setminus \{0\}$.
\end{enumerate}
Then, $G(P) \in \L(X)$, is invertible with bounded inverse.
\end{Lem}

Now, to inverse the determinant of system \eqref{syst trans M}, we introduce some holomorphic functions which will play the role of $G$ in the previous lemma. Then, we study them on the positive real axis. 

Let $\delta > 0$ and $z \in \CC\setminus \RR_-$. We set
$$\left\{\begin{array}{lll}
u_\delta (z) & = & \dis 1 - e^{-2\delta\sqrt{z}} - 2\delta \sqrt{z} e^{-\delta\sqrt{z}} \\ \ecart

v_\delta (z) & = & \dis 1 - e^{-2\delta\sqrt{z}} + 2\delta \sqrt{z} e^{-\delta\sqrt{z}}.
\end{array}\right.$$
Then, we have
$$\begin{array}{lllll}
U_+ = u_d(-A), & U_- = u_c(-A), & V_+ = v_d(-A) & \text{and} & V_- = v_c(-A).
\end{array}$$ 
We set 
$$\CC_+ = \{z \in \CC : \text{Re}(z) \geqslant 0\}.$$
Moreover, from \cite{thorel}, Lemma 4.4, we have $u_\delta (z) \neq 0$ and $v_\delta (z) \neq 0$ for $z \in \CC_+ \setminus \{0\}$, thus we note 
$$\left\{\begin{array}{ccl}
f_{\delta,1} (z) & = & \dis u_\delta^{-1} (z) \left(1 + e^{-\delta\sqrt{z}}\right)^2 + v_\delta^{-1}(z) \left(1 - e^{-\delta\sqrt{z}}\right)^2 \\ \ecart

f_{\delta,2} (z) & = & \dis \left(u_\delta^{-1}(z) + v_\delta^{-1}(z) \right) \left(1 - e^{-2\delta\sqrt{z}}\right) \\ \ecart

f_{\delta,3} (z) & = & \dis u_\delta^{-1}(z) \left(1 - e^{-\delta\sqrt{z}}\right)^2 + v_\delta^{-1}(z) \left(1 + e^{-\delta\sqrt{z}}\right)^2 \\ \ecart

g_\delta (z) & = & \dis 16\, u_\delta^{-1}(z) v_\delta^{-1}(z) e^{-2\delta\sqrt{z}}.
\end{array}\right.$$
Thus, we obtain that
$$\left\{\begin{array}{lll}
P_1^+ = k_+f_{d,1}(-A), & P_2^+ = k_+f_{d,2}(-A), & P_3^+ = k_+f_{d,3}(-A) \\ \ecart
P_1^- = k_-f_{c,1}(-A), & P_2^- = k_-f_{c,2}(-A), & P_3^- = k_-f_{c,3}(-A).
\end{array}\right.$$
Moreover, for $z \in \CC_+\setminus \{0\}$, we define
$$
f(z) := k_+^2 g_d(z) + k_-^2 g_c(z) + k_+k_- \left(f_{d,1} (z) f_{c,3} (z) + f_{c,1} (z)f_{d,3} (z) + 2 f_{d,2} (z)f_{c,2} (z)\right).
$$
Note that $f \in H(S_\theta)$, for all $\theta \in (0,\pi)$. It follows that 
\begin{equation}\label{f(-A)}
f(-A) = 16k_+^2 U_+^{-1}V_+^{-1} e^{2dM} + 16k_-^2 U_-^{-1}V_-^{-1} e^{2cM} + P_1^+P_3^- + P_1^-P_3^+ + 2 P_2^+ P_2^-.
\end{equation}

\begin{Lem} \label{lem det pas 0 M}
For all $x \in \RR_+ \setminus \{0\}$, then $f(x)$ does not vanish.
\end{Lem}
\begin{proof}
Let $\delta>0$, from Lemma 4.4, in \cite{thorel}, for all $z \in \CC_+ \setminus \{0\}$, we have 
$$ 0 < |u_\delta (z)| \quad \text{and} \quad 0 < |v_\delta (z)|.$$
Thus, for all $x > 0$, one has $u_\delta (x)$, $v_\delta (x) > 0$. Moreover, for all $x > 0$, we deduce that
$$f_{\delta,1} (x), f_{\delta,2} (x), f_{\delta,3} (x), g_\delta(x) > 0$$ 
and since $k_+k_- > 0$, it follows that
$$ f(x) = k_+^2 g_d(x) + k_-^2 g_c(x) + k_+k_- \left(f_{d,1} (x) f_{c,3} (x) + f_{c,1} (x)f_{d,3} (x) + 2 f_{d,2} (x)f_{c,2} (x)\right) > 0.$$
\end{proof}

\section{Proof of the main result}\label{Section proof main res}

Assume that (P) has a unique classical solution, then, \eqref{reg phi +-} holds from \refT{Th syst trans M}. \\
Conversely, Assume that \eqref{reg phi +-} holds. From \refT{Th syst trans M}, it suffices to show that system \eqref{syst trans M} has a unique solution such that \eqref{reg (psi1,psi2) M} is satisfies. 

The proof is divided in three parts. In the first part, we make explicit the determinant of system \eqref{syst trans M}. Then, in the second part, we show the uniqueness of the solution, to this end, we inverse the determinant with the help of functional calculus. Finally, in the last part, we prove that $\psi_1$ and $\psi_2$ have the expected regularity. 

\subsection{Calculus of the determinant}

In this section, we make explicit the determinant. Recall system \eqref{syst trans M}:
$$
\left\{\begin{array}{lcl}
\left(P_1^+ + P_1^-\right) M \psi_1 - \left(P_2^+ - P_2^-\right) \psi_2 & = & S_1 \\ \ecart

\left(P_2^+ - P_2^-\right) M \psi_1 - \left(P_3^+ + P_3^-\right) \psi_2 & = & S_2.
\end{array}\right.
$$
Moreover, we writes the previous system as the matrix equation $\Lambda \Psi = S$, where
$$\Lambda = \left(\begin{array}{cc}
M \left(P_1^+ + P_1^-\right) &  - \left(P_2^+ - P_2^-\right) \\
M \left(P_2^+ - P_2^-\right) &  - \left(P_3^+ + P_3^-\right)
\end{array}\right), \quad \Psi = \left(\begin{array}{c}
\psi_1 \\
\psi_2
\end{array}\right), \quad S = \left(\begin{array}{c}
S_1 \\
S_2
\end{array}\right).$$
To solve system \eqref{syst trans M}, we calculate the determinant of the associated matrix $\Lambda$:
$$\begin{array}{lll}
\det(\Lambda) & := &\dis - M \left(P_1^+ + P_1^-\right) \left(P_3^+ + P_3^-\right) + M \left(P_2^+ - P_2^-\right)^2 \\ \ecart

& = & \dis - M \left(P_1^+P_3^+ + P_1^+P_3^- + P_1^-P_3^+ + P_1^-P_3^- - \left(P_2^+\right)^2 - \left(P_2^-\right)^2 + 2 P_2^+ P_2^- \right) \\ \ecart

& = & \dis - M \left(P_1^+P_3^+ - \left(P_2^+\right)^2  + P_1^-P_3^- - \left(P_2^-\right)^2 + P_1^+P_3^- + P_1^-P_3^+ + 2 P_2^+ P_2^- \right).
\end{array}$$
We begin to calculate each terms separately. Then, we have
$$\begin{array}{lll}
P_1^+P_3^+ & = &\dis k_+^2 \left(U_+^{-1} \left(I + e^{dM}\right)^2 + V_+^{-1} \left(I - e^{dM}\right)^2\right) \left(U_+^{-1}\left(I - e^{dM}\right)^2 + V_+^{-1} \left(I + e^{dM}\right)^2 \right) \\ \ecart

& = & \dis k_+^2 \left( \left(U_+^{-2} + V_+^{-2} \right) \left(I - e^{2dM}\right)^2 + U_+^{-1}V_+^{-1} \left( \left(I + e^{dM}\right)^4 + \left(I - e^{dM}\right)^4\right) \right)
\end{array}$$
and
$$\left(I + e^{dM}\right)^4 + \left(I - e^{dM}\right)^4 = 2\left(I + e^{2dM}\right)^2 + 8e^{2dM}.$$
Thus
$$P_1^+P_3^+ = k_+^2 \left(\left(U_+^{-2} + V_+^{-2} \right) \left(I - e^{2dM}\right)^2 + 2U_+^{-1}V_+^{-1}\left(I + e^{2dM}\right)^2 + 8U_+^{-1}V_+^{-1}e^{2dM}\right).$$
Moreover, we have
$$\begin{array}{lll}
\left(P_2^+\right)^2 & = & \dis k_+^2 \left(U_+^{-1} + V_+^{-1}\right)^2 \left(I - e^{2dM}\right)^2 \\ \ecart 
& = & \dis k_+^2 \left(U_+^{-2} + V_+^{-2}\right) \left(I - e^{2dM}\right)^2 + 2k_+^2 U_+^{-1}V_+^{-1}\left(I - e^{2dM}\right)^2.
\end{array}$$
It follows 
$$\begin{array}{lll}
P_1^+P_3^+ - \left(P_2^+\right)^2 & = & \dis k_+^2 \left(\left(U_+^{-2} + V_+^{-2} \right) \left(I - e^{2dM}\right)^2 + 2U_+^{-1}V_+^{-1}\left(I + e^{2dM}\right)^2 \right) \\ \ecart 
&& + 8 k_+^2 U_+^{-1}V_+^{-1}e^{2dM} \\ \ecart
&& \dis - k_+^2 \left(U_+^{-2} + V_+^{-2}\right) \left(I - e^{2dM}\right)^2 - 2k_+^2 U_+^{-1}V_+^{-1}\left(I - e^{2dM}\right)^2 \\ \\

& = & \dis 2k_+^2 U_+^{-1}V_+^{-1} \left(\left(I + e^{2dM}\right)^2 - \left(I - e^{2dM}\right)^2+ 4 e^{2dM}\right) \\ \\

& = & \dis 2k_+^2 U_+^{-1}V_+^{-1} \left(4e^{2dM} + 4 e^{2dM}\right) \\ \\

& = & \dis 16k_+^2 U_+^{-1}V_+^{-1} e^{2dM}.
\end{array}$$
In the same way, replacing respectively $k_+$, $U_+^{-1}$, $V_+^{-1}$ and $d$ by $k_-$, $U_-^{-1}$, $V_-^{-1}$ and $c$, we obtain 
$$P_1^-P_3^- - \left(P_2^-\right)^2 = 16k_-^2 U_-^{-1}V_-^{-1} e^{2cM}.$$
Thus, the determinant of $\Lambda$ writes
\begin{equation}\label{det M}
\det(\Lambda) = - M \left(16k_+^2 U_+^{-1}V_+^{-1} e^{2dM} + 16k_-^2 U_-^{-1}V_-^{-1} e^{2cM} + P_1^+P_3^- + P_1^-P_3^+ + 2 P_2^+ P_2^- \right).
\end{equation}
Finally, from \eqref{f(-A)} and \eqref{det M}, we obtain 
\begin{equation}\label{det = f(-A) M}
\det(\Lambda) = -M f(-A).
\end{equation}
Note that, since $f(-A) \in \L(X)$, it follows that $D(\det(\Lambda)) = D(M)$.

\subsection{Inversion of the determinant}

Let $C_1$, $C_2$ two linear operators in $X$. We note $C_1 \sim C_2$ to means that \mbox{$C_1 = C_2 + \Sigma$}, where $\Sigma$ is a finite sum of term of type $k M^n e^{\alpha M}$, with $k\in \RR$, $n\in \NN$ and $\alpha \in \RR_+\setminus\{0\}$. Note that $\Sigma$ is a regular term in the sense: 
$$\Sigma \in \L(X) \quad\text{with}\quad \Sigma(X) \subset D(M^\infty):= \bigcap_{k\geqslant0}D(M^k).$$

Since $U_\pm \sim I$ and $V_\pm \sim I$, then setting $W = U_+ U_- V_+ V_- \sim I$, we obtain
$$\left\{\begin{array}{lcllcl}
WP_1^+ & \sim & 2k_+I, & WP_1^- & \sim & 2k_-I, \\ \ecart
WP_2^+ & \sim & 2k_+I, & WP_2^- & \sim & 2k_-I, \\ \ecart
WP_3^+ & \sim & 2k_+I, & WP_3^- & \sim & 2k_-I.
\end{array}\right.$$
Moreover, we have
$$16k_+^2 W^2 U_+^{-1}V_+^{-1} e^{2dM} \sim 0 \quad \text{and} \quad 16k_-^2 W^2 U_-^{-1}V_-^{-1} e^{2cM} \sim 0.$$
We then deduce the following relation 
$$\begin{array}{lll}
\dis - M^{-1} W^2 \det(\Lambda) & = & \dis 16k_+^2 W^2 U_+^{-1}V_+^{-1} e^{2dM} + 16k_-^2 W^2 U_-^{-1}V_-^{-1} e^{2cM} \\ \ecart
&& \dis + WP_1^+ WP_3^- + WP_1^- WP_3^+ + 2 WP_2^+ WP_2^- \\ \\
& \sim & \dis 4 k_+k_-I + 4 k_+k_-I + 8 k_+k_-I = 16 k_+k_-I. 
\end{array}$$
Thus, we get 
\begin{equation}\label{det 2 M}
\det(\Lambda) = -W^{-2}M\left(16k_+k_-I + \sum_{j\in J} k_j M^{n_j} e^{\alpha_j M}\right),
\end{equation}
where $J$ is a finite set and for all $j \in J$:
$$k_j \in \RR, n_j \in \NN \quad \text{and} \quad \alpha_j \in \RR_+\setminus\{0\}.$$
From \eqref{det 2 M}, we have
\begin{equation}\label{det 3 M}
\det(\Lambda) = -16 k_+k_- W^{-2}M F, 
\end{equation}
where
\begin{equation}\label{F M}
F = I + \sum_{j\in J} \frac{k_j}{16 k_+k_-} M^{n_j} e^{\alpha_j M}.
\end{equation}
For $z \in \CC \setminus \RR_-$, we set
$$\tilde{f}(z) = 1 + \sum_{j\in J} \frac{k_j}{16 k_+k_-} \left(-\sqrt{z}\right)^{n_j} e^{-\alpha_j \sqrt{z}}.$$
Then, $F = \tilde{f}(-A)$ and from \eqref{det = f(-A) M} and \eqref{det 3 M}, we have
$$f(-A) = -M^{-1}\det(\Lambda) = 16 k_+k_- W^{-2} \tilde{f}(-A).$$
Thus, by construction, for $z \in \CC \setminus \RR_-$, the link between $f$ and $\tilde{f}$ is 
\begin{equation}\label{f(z) = tilde f (z) M}
f(z) = 16 k_+k_- u_d^{-2}(z)u_c^{-2}(z)v_d^{-2}(z)v_c^{-2}(z) \tilde{f}(z).
\end{equation}
\begin{Prop}\label{prop inv F M}
The operator $F \in \L(X)$ defined by \eqref{F M}, is invertible with bounded inverse.
\end{Prop}
\begin{proof}
Note that $f,\tilde{f} \in H(S_\theta)$, for a given $\theta \in (0,\pi)$. Moreover, for $z \in \CC\setminus\RR_-$ and $j\in J$, functions $\dis\frac{k_j}{16 k_+k_-} \left(-\sqrt{z}\right)^{n_j}$ are polynomial. Thus, $1-\tilde{f} \in \mathcal{E}_\infty (S_\theta)$.

From Lemma 4.4 in \cite{thorel}, $u_d$, $u_c$, $v_d$ and $v_c$ do not vanish on $\CC_+\setminus\{0\}$. Moreover, due to \refL{lem det pas 0 M}, $f$ does not vanish on $\RR_+ \setminus\{0\}$, Thus, we deduce from \eqref{f(z) = tilde f (z) M} that $\tilde{f}$ do not vanish on $\RR_+ \setminus\{0\}$. 

Finally, we apply \refL{lem inv} with $P = -A$ and $G = \tilde{f}$ to obtain that $F = \tilde{f}(-A) \in \L(X)$ is invertible with bounded inverse. 
\end{proof}

We are now in position to prove the main result of this section.

\begin{Prop}\label{prop inv det M}
The operator $\det(\Lambda)$, defined by \eqref{det 3 M} is invertible with bounded inverse.
\end{Prop}
\begin{proof}
From \refL{lem U+- V+- inv}, $U_+$, $U_-$, $V_+$ and $V_-$ are bounded invertible operators with bounded inverse. So we deduce that $W^{-2}$ is invertible with bounded inverse. Moreover, from $(H_2)$, \eqref{det 3 M} and \refP{prop inv F M}, we obtain that $\det(\Lambda)$ is invertible with bounded inverse. 
\end{proof}

\subsection{Regularity}

To study the regularity, we need to recall the following technical result from \cite{LMMT}, Lemma 5.1.
\begin{Lem}[\cite{LMMT}]\label{lemreg}
Let $V \in \L(X)$ such that $\dis 0 \in \rho(I+V)$. Then, there exists \mbox{$W \in \L(X)$} such that $$(I+V)^{-1} = I - W,$$ 
and $W(X) \subset V(X)$. Moreover, if $T$ is a linear operator in $X$ such that $V(X) \subset D(T)$ and for $\psi \in D(T)$, \mbox{$TV \psi=VT \psi$}, then $WT \psi = TW \psi$.
\end{Lem}
From \refP{prop inv det M}, system \eqref{syst trans M} has a unique solution $(\psi_1,\psi_2)$. From \refT{Th syst trans M}, it remains to prove that
$$\psi_1 \in \left(D(A), X\right)_{1+\frac{1}{2p},p} \quad \text{and} \quad \psi_2 \in \left(D(A), X\right)_{1+\frac{1}{2}+\frac{1}{2p},p}.$$
To this end, we have to study the regularity of the inverse of the determinant $\det(\Lambda)$.
\begin{Lem}
There exists $R_d \in D(M^\infty)$ such that
$$\left[\det(\Lambda)\right]^{-1} = -\frac{1}{16k_+k_-} M^{-1} + R_d.$$
\end{Lem}
\begin{proof}
From \eqref{det 3 M} and \refP{prop inv det M}, we have
$$\left[\det(\Lambda)\right]^{-1} = -\frac{1}{16k_+k_-} M^{-1} W^2 F^{-1}.$$
From \refL{lemreg}, there exists $R_F \in D(M^\infty)$, such that
$$F^{-1} = I + R_F.$$
Moreover, for $\delta>0$, we know that
\begin{equation}\label{expMinfty M}
e^{\delta M} \psi \in D(M^\infty).
\end{equation}
Since $W = U_+ U_- V_+ V_-$, from \eqref{expMinfty M}, there exists $R_W \in D(M^\infty)$ such that
$$W^2 = I + R_W.$$
We deduce that there exists $R_d \in D(M^\infty)$, such that
$$\left[\det(\Lambda)\right]^{-1} = -\frac{1}{16k_+k_-} M^{-1} + R_d.$$
\end{proof}
Now, we study the regularity of $\psi_1$ and $\psi_2$. We recall that
$$\Lambda \Psi = S,$$
where $\Lambda$ is invertible from \refP{prop inv det M}. From \refL{lemreg}, there exist \mbox{$R_{U_\pm},R_{V_\pm} \in D(M^\infty)$}, such that
\begin{equation}\label{Upm-1 Vpm-1 M}
U_\pm^{-1}=I+R_{U_\pm}\qquad\text{and}\qquad V_\pm^{-1}=I+R_{V_\pm}.
\end{equation}
From \eqref{expMinfty M} and \eqref{Upm-1 Vpm-1 M}, there exist $R_i \in D(M^\infty)$, $i=1,2,3,4$, such that
$$\Lambda = \left(\begin{array}{cc}
2(k_+ + k_-)M + R_1 & -2(k_+ - k_-)I + R_2 \\ \ecart
2(k_+ - k_-)M + R_3 & -2(k_+ + k_-)I + R_4
\end{array}\right).$$
It follows that there exist $\R_1$, $\R_2 \in D(M^\infty)$, such that
\begin{equation}\label{syst psi1 psi2 M}
\left\{\begin{array}{lll}
\psi_1 & = & \dis \frac{(k_+ + k_-)}{8 k_+ k_-} M^{-1} S_1 - \frac{(k_+ - k_-)}{8 k_+ k_-} M^{-1} S_2 + \R_1 \\ \ecart

\psi_2 & = & \dis \frac{(k_+ - k_-)}{8 k_+ k_-} S_1 - \frac{(k_+ + k_-)}{8 k_+ k_-} S_2 + \R_2,
\end{array}\right.
\end{equation}
where $S_1$ is given by \eqref{S1 M} and $S_2$ is given by \eqref{S2 M}.

Since $F_+$ is a classical solution of problem $(P_+)$ and $F_-$ is a classical solution of problem $(P_-)$, then from \refR{Rem al- M} and \refR{Rem al+ M}, we obtain that $\check{S}$ defined by \eqref{R M} has the following regularity
\begin{equation}\label{reg R M}
\check{S} \in (D(M),X)_{\frac{1}{p},p}.
\end{equation}
Moreover, from \eqref{egalites espaces interpolation}, we have
$$\varphi_1^+, \varphi_1^- \in \left(D(A), X\right)_{1+\frac{1}{2p},p} = \left(D(M), X\right)_{3+\frac{1}{p},p}$$
and
$$\varphi_2^+, \varphi_2^- \in \left(D(A), X\right)_{1+\frac{1}{2}+\frac{1}{2p},p} = \left(D(M), X\right)_{2+\frac{1}{p},p}.$$
Thus, from \eqref{phi i tilde - M}, \eqref{phi i tilde + M}, \refL{lemreg}, \refR{Rem al- M} and \refR{Rem al+ M}, we deduce that
\begin{equation}\label{reg phi i tilde +- M}
\tilde{\varphi}_1^+,\tilde{\varphi}_1^-,\tilde{\varphi}_3^+,\tilde{\varphi}_3^- \in \left(D(M), X\right)_{3+\frac{1}{p},p} \quad \text{and} \quad \tilde{\varphi}_2^+,\tilde{\varphi}_2^-,\tilde{\varphi}_4^+,\tilde{\varphi}_4^- \in \left(D(M), X\right)_{2+\frac{1}{p},p}.
\end{equation}
So, from \eqref{S1 M}, \eqref{S2 M}, \eqref{expMinfty M}, \eqref{reg R M} and \eqref{reg phi i tilde +- M}, we have
\begin{equation}\label{reg S1 S2 M}
S_1 \in (D(M),X)_{2 + \frac{1}{p},p} \quad \text{and} \quad S_2 \in (D(M),X)_{2 + \frac{1}{p},p}.
\end{equation}
Finally, from \eqref{egalites espaces interpolation}, \eqref{syst psi1 psi2 M} and \eqref{reg S1 S2 M}, we obtain 
$$\psi_1 \in (D(M),X)_{3 + \frac{1}{p},p} = \left(D(A), X\right)_{1+\frac{1}{2p},p} \quad \text{and} \quad \psi_2 \in (D(M),X)_{2 + \frac{1}{p},p} = \left(D(A), X\right)_{1+\frac{1}{2}+\frac{1}{2p},p}.$$

\section*{Acknowledgments} 

This research is supported by CIFRE contract 2014/1307 with Qualiom Eco company and partially by the LMAH and the european funding ERDF through grant project Xterm. 

\noindent I would like to thank the referee for the valuable comments and corrections which help to improve this paper.


\begin{thebibliography}{biblio}

\bibitem{JMAA 1} 
 \newblock B. Barraza Mart\'inez, R. Denk, J. Hern\'andez Monz\'on, F. Kammerlander and M. Nendel,
 \newblock Regularity and asymptotic behavior for a damped plate–membrane transmission problem, 
 \newblock \emph{J. Math. Anal. Appl}, \textbf{474} (2) (2019), 1082-1103. 

\bibitem{bourgain}
 \newblock J.~Bourgain, 
 \newblock Some remarks on Banach spaces in which martingale difference sequences are unconditional,  
 \newblock \emph{Ark. Mat.}, \textbf{21} (1983), 163-168. 

\bibitem{burkholder} 
 \newblock D. L.~Burkholder, 
 \newblock A geometrical characterisation of Banach spaces in which martingale difference sequences are unconditional, 
 \newblock \emph{Ann. Probab.}, \textbf{9} (1981), 997-1011. 

\bibitem{CV 1} 
 \newblock F. Cakoni, G. C. Hsiao and W. L. Wendland, 
 \newblock On the boundary integral equation methodfor a mixed boundary value problem of the biharmonic equation, 
 \newblock \emph{Complex variables}, \textbf{50} (7-11) (2005), 681-696. 

\bibitem{cohen-murray} 
 \newblock D. S. Cohen and J. D. Murray, 
 \newblock A generalized diffusion model for growth and dispersal in population, 
 \newblock \emph{Journal of Mathematical Biology}, \textbf{12} (1981), 237-249. 

\bibitem{elasticity 1} 
 \newblock M. Costabel, E. Stephan and W. L. Wengland, 
 \newblock On boundary integral equations of the first kind for the bi-Laplacian in a polygonal plane domain, 
 \newblock \emph{Ann. Scuola Norm. Sup. Pisa Cl. Sci.}, $4^{\text{e}}$ serie, \textbf{10} (2) (1983), 197-241. 

\bibitem{da prato-grisvard} 
 \newblock G. Da Prato and P. Grisvard, 
 \newblock Sommes d'op\'erateurs lin\'eaires et \'equations diff\'erentielles op\'erationnelles, 
 \newblock \emph{J. Math. pures et appl.}, \textbf{54} (1975), 305-387.

\bibitem{dore-venni} 
 \newblock G. Dore and A. Venni, 
 \newblock On the closedness of the sum of two closed operators, 
 \newblock \emph{Math. Z.}, \textbf{196} (1987), 189-201. 

\bibitem{hassan} 
 \newblock A. Favini, R. Labbas, K. Lemrabet, S. Maingot and H. Sidib\'e, 
 \newblock Transmission Problem for an Abstract Fourth-order Differential Equation of Elliptic Type in UMD Spaces, 
 \newblock \emph{Advances in Differential Equations}, \textbf{15} (1-2) (2010), 43-72. 

\bibitem{menad} 
 \newblock A. Favini, R. Labbas, A. Medeghri and A. Menad, 
 \newblock Analytic semigroups generated by the dispersal process in two habitats incorporating individual behavior at the interface, 
 \newblock \emph{J. Math. Anal. Appl.}, \textbf{471} (2019), 448–480. 

\bibitem{gilbarg-trudinger}
 \newblock D. Gilbarg and N. Trudinger, 
 \newblock \emph{Elliptic partial differential equations of second order}, 
 \newblock Classics in Mathematics, Springer-Verlag, Berlin, 2001. 

\bibitem{grisvard} 
 \newblock P. Grisvard, 
 \newblock Spazi di tracce e applicazioni, 
 \newblock \emph{Rendiconti di Mat.}, Serie VI, \textbf{5} (4) (1972), 657-729.

\bibitem{electrostatic 1}
 \newblock Z. Guo, B. Lai and D. Ye, 
 \newblock Revisiting the biharmonic equation modelling electrostatic actuation in lower dimensions, 
 \newblock \emph{Proc. Amer. Math. Soc.}, \textbf{142} (2014), 2027-2034. 

\bibitem{haase}
 \newblock M. Haase, 
 \newblock \emph{The functional calculus for sectorial Operators}, 
 \newblock Birkhauser, 2006. 

\bibitem{JMAA 2} 
 \newblock F. Hassine, 
 \newblock Logarithmic stabilization of the Euler–Bernoulli transmission plate equation with locally distributed Kelvin–Voigt damping, 
 \newblock \emph{J. Math. Anal. Appl.}, \textbf{455} (2) (2017), 1765-1782. 

\bibitem{elasticity 2}
 \newblock A. F. Hrustalev and B. I. Kogan,
 \newblock A boundary-value problem for the biharmonic equation in elasticity theory, 
 \newblock \emph{(Russian) Izv. Vysš. Učebn. Zaved. Matematika}, \textbf{4} (3) (1958), 241-247. 

\bibitem{komatsu} 
 \newblock H. Komatsu, 
 \newblock Fractional powers of operators, 
 \newblock \emph{Pac. J. Math.}, \textbf{19} (2) (1966), 285-346. 

\bibitem{JEE 2} 
 \newblock M. Kotschote, 
 \newblock Maximal $L^p$-regularity for a linear three-phase problem of para-bolic–elliptic type, 
 \newblock \emph{Journal of Evolution Equations}, \textbf{10} (2) (2010), 293-318. 

\bibitem{LMMT} 
 \newblock R. Labbas, S. Maingot, D. Manceau and A. Thorel, 
 \newblock On the regularity of a generalized diffusion problem arising in population dynamics set in a cylindrical domain, 
 \newblock \emph{J. Math. Anal. Appl.}, \textbf{450} (2017), 351-376. 

\bibitem{LLMT} 
 \newblock R. Labbas, K. Lemrabet, S. Maingot and A. Thorel, 
 \newblock Generalized linear models for population dynamics in two juxtaposed habitats, 
 \newblock \emph{Discrete Contin. Dyn. Syst. - A}, \textbf{39} (5) (2019), 2933-2960. 

\bibitem{maelis} 
 \newblock K. Limam, R. Labbas, K. Lemrabet, A. Medeghri and M. Meisner, 
 \newblock On Some Transmission Problems Set in a Biological Cell, Analysis and Resolution, 
 \newblock \textit{J. Differential Equations}, \textbf{259} (7) (2015), 2695-2731. 

\bibitem{electrostatic 2}
 \newblock F. Lin and Y. Yang,
 \newblock Nonlinear non-local elliptic equation modelling electrostatic actuation, 
 \newblock \emph{Proceedings of the Royal Society of London A}, \textbf{463} (2007), 1323-1337. 

\bibitem{lions-peetre} 
 \newblock J.-L. Lions and J. Peetre, 
 \newblock Sur une classe d'espaces d'interpolation, 
 \newblock \emph{Publications math\'e-matiques de l'I.H.\'E.S.}, \textbf{19} (1964), 5-68. 

\bibitem{lunardi}
 \newblock A.~Lunardi, 
 \newblock \emph{Analytic semigroups and optimal regularity in parabolic problems}, 
 \newblock Birkhauser, Basel, Boston, Berlin, 1995. 

\bibitem{ochoa} 
 \newblock F. L. Ochoa, 
 \newblock A generalized reaction-diffusion model for spatial structures formed by motile cells, 
 \newblock \emph{BioSystems}, \textbf{17} (1984), 35-50. 

\bibitem{perfekt} 
 \newblock K.-M. Perfekt, 
 \newblock The transmission problem on a three-dimensional wedge, 
 \newblock \emph{Arch. Ration. Mech. Anal.}, \textbf{231} (3) (2019), 1745-1780. 

\bibitem{pruss-sohr} 
 \newblock J. Pr\"uss and H. Sohr, 
 \newblock On operators with bounded imaginary powers in Banach spaces, 
 \newblock \emph{Mathematische Zeitschrift}, \textbf{203} (1990), 429-452. 

\bibitem{pruss-sohr2} 
 \newblock J. Pr\"uss and H. Sohr, 
 \newblock Imaginary powers of elliptic second order differential operators in $L^p$-spaces, 
 \newblock \emph{Hiroshima Math. J.}, \textbf{23} (1) (1993), 161-192. 

\bibitem{rubio} 
 \newblock J. L. Rubio de Francia, 
 \newblock Martingale and integral transforms of Banach space valued functions, 
 \newblock \emph{Probability and Banach Spaces: Lecture Notes in Math.}, \textbf{1221} (1986), 195-222.

\bibitem{IJPAM} 
 \newblock H. Saker and N. Bouselsal, 
 \newblock On the bilaplacian problem with nonlinear boundary conditions, 
 \newblock \emph{Indian J. Pure Appl. Math.}, \text{47} (3) (2016), 425–435. 


\bibitem{thorel} A. Thorel, 
 \newblock Operational approach for biharmonic equations in $L^p$-spaces, 
 \newblock \emph{Journal of Evolution Equations}, \textbf{20} (2020), 631-657.

\bibitem{triebel} 
 \newblock H. Triebel, 
 \newblock \emph{Interpolation theory, function Spaces, differential Operators},
 \newblock North-Holland publishing company Amsterdam New York Oxford, 1978.

\bibitem{DCDS-B} 
 \newblock J. Wang, J. Lang and Y. Chen, 
 \newblock Global dynamics of an age-structured HIV infection model incorporating latency and cell-to-cell transmission, 
 \newblock \emph{Discrete Contin. Dyn. Syst. - B}, \textbf{22} (10) (2017), 3721-3747. 

\bibitem{JDE} 
 \newblock C.-F. Yanga and S. Buterin, 
 \newblock Uniqueness of the interior transmission problem with partial information on the potential and eigenvalues, 
 \newblock \emph{J. Differential Equations}, \textbf{260} (6) (2016), 4871-4887. 

\end{thebibliography}

\end{document}